\documentclass[11pt]{article}
\usepackage{amsmath, amssymb, amsthm, verbatim,enumerate,bbm, mathtools}
\usepackage{indentfirst}

\date{\today}
\allowdisplaybreaks

\theoremstyle{plain}
\newtheorem{thm}{Theorem}
\newtheorem{conj}[thm]{Conjecture}
\newtheorem{corol}[thm]{Corollary}
\newtheorem{theorem}{Theorem}[section]
\newtheorem{lemma}[theorem]{Lemma}

\title{Every graph contains a linearly sized induced subgraph with all degrees odd}
\author{Asaf Ferber\thanks{Department of Mathematics, University of California, Irvine.
Email:asaff@uci.edu.
Research supported in part by NSF grants DMS-1954395 and DMS-1953799.} \and Michael Krivelevich\thanks{School of Mathematical Sciences, Tel Aviv University, Tel Aviv 6997801, Israel. Email: krivelev@tauex.tau.ac.il.
Research supported in part by USA--Israel BSF grant 2018267 and by ISF grant 1261/17.}}

\begin{document}

\maketitle
\begin{abstract}
  We prove that every graph $G$ on $n$ vertices with no isolated vertices contains an induced subgraph of size at least $n/10000$ with all degrees odd. This solves an old and well-known conjecture in graph theory.
\end{abstract}

\section{Introduction}

We start with recalling a classical theorem of Gallai (see \cite{Lovasz}, Problem 5.17 for a proof):
\begin{thm}
[Gallai's Theorem]  \label{gallai} Let $G$ be any graph.
\begin{enumerate}
    \item There exists a partition $V(G)=V_1\cup V_2$ such that both graphs $G[V_1]$ and $G[V_2]$ have all their degrees even.
    \item There exists a partition $V(G)=V_o\cup V_e$ such that the graph $G[V_e]$ has all its degrees even, and the graph $G[V_o]$ has all its degrees odd.
\end{enumerate}
\end{thm}
 It follows immediately from $1.$ that every graph $G$ has an induced subgraph of size at least $|V(G)|/2$ with all its degrees even. This is easily seen to be tight by taking $G$ to be a path.

It is natural to ask whether we can derive analogous results for induced subgraphs with
all degrees odd. Some caution is required here --- an isolated vertex can never be a part of a subgraph with all degrees odd. Thus we restrict our attention to graphs of positive minimum degree.

Let us introduce a relevant notation: given a graph $G=(V,E)$, we define
$$
f_o(G) = \max\{|V_0|: G[V_0]\mbox{ has all degrees odd.}\},
$$
and set
$$
f_o(n)=\min\{f_o(G) \mid G \textrm{ is a graph on }n \textrm{ vertices with } \delta(G)\geq 1\}.
$$

The following is a very well known conjecture, aptly described by Caro already more than a quarter century ago \cite{Car94} as ``part of the graph theory folklore":
\begin{conj}
\label{conje}
There exists a constant $c>0$ such that for every $n\in \mathbb{N}$ we have $f_o(n)\geq cn$.
\end{conj}

Caro himself proved \cite{Car94}  that $f_o(n)=\Omega(\sqrt{n})$, resolving a question of Alon who asked whether $f_o(n)$ is polynomial in $n$.  The current best bound, due to Scott \cite{Sco92}, is $f_o(n)=\Omega(n/\log n)$. There have been numerous variants and partial results about the conjecture, we will not cover them here.

Our main result establishes Conjecture \ref{conje} with $c=0.0001$.

\begin{thm}  \label{main}
Every graph $G$ on $n$ vertices with $\delta(G)\geq 1$ satisfies: $f_o(G)\ge cn$ for $c=\frac{1}{10000}$.
\end{thm}
With some effort/more accurate calculations the constant can be improved but probably to a value which is still quite far from the optimal one; we decided not to invest a substantial effort in its optimization and just chose some constants that work.

A relevant parameter was studied by Scott \cite{Sco01}: given a graph $G$ with no isolated vertices, let $t(G)$ be the minimal $k$ for which there exists a vertex cover of $G$ with $k$ sets, each spanning an induced graph with all degrees odd. Letting
$$
t(n)=\min\{t(G) \mid G \textrm{ is a graph on } n \textrm{ vertices with }\delta(G)\geq 1\},
$$
Scott proved (Theorem 4 in \cite{Sco01}) that
$$\Omega(\log n)=t(n)=O(\log^2n).$$

As indicated by Scott already, showing that $f_o(n)$ is linear in $n$ proves the following:

\begin{corol}
$t(n)=\Theta(\log n).$
\end{corol}

For completeness, we outline its proof here.

\begin{proof}
Let $G$ be a graph on $n$ vertices with $\delta(G)\geq 1$. By a repeated use of Theorem \ref{main}, we can find disjoint sets $V_1,\ldots,V_t$ such that:
\begin{enumerate}
    \item $V_i\subseteq V(G)\setminus \left(\bigcup_{j=1}^{i-1}V_j\right)$, and
    \item all the degrees in $G[V_i]$ are odd, and
    \item letting $n_i$ be the number of non-isolated vertices in $G\left[ V(G)\setminus \left(\bigcup_{j=1}^{i-1}V_j\right)\right]$, we have that $|V_i|\geq n_i/10000$.
\end{enumerate}

We continue the above process as long as $n_i>0$. Clearly, the process terminates after $t=O(\log n)$ steps. Moreover, letting $U=V(G)\setminus \left(\bigcup_{i=1}^tV_i\right)$, we have that $U$ is an independent set in $G$. Finally, as shown in the proof of Theorem 4 in \cite{Sco01}, every independent set in such $G$ can be covered by $O(\log n)$ odd graphs. This proves that $t(n)=O(\log n)$.

To show a lower bound, we can use the following example due to Scott \cite{Sco01}: assume $n$ is of the form $n=s+\binom{s}{2}$. Let the vertex set of $G$ be composed of two disjoint sets: $A$ of size $s$ associated with $[s]$, and $B$ of size $\binom{s}{2}$ associated with $\binom{[s]}{2}$. The graph $G$ is bipartite with the edges defined as follows: a pair $(i,j)\in B$ is connected to both $i,j\in A$. Observe that if $U\subset V(G)$ spans a subgraph of $G$ with all degrees odd and containing $(i,j)\in B$, then $U$ contains exactly one of $i,j\in A$. Hence if ${\cal U}=(U_1,\ldots, U_t)$ forms a cover of $V(G)$ with subsets spanning odd subgraphs, then ${\cal U}$ separates the set $A$, and the minimum size of such a separating family is easily shown to be asymptotic to $\log_2s=\Omega(\log_2 n)$.
\end{proof}

\section{Auxiliary results}

The following lemma appears as Theorem 2.1 in \cite{Car94}. For the convenience of the reader we provide its simple proof.

\begin{lemma}
\label{lemma1}
For every graph $G$ we have that
$f_o(G)\ge \frac{\Delta(G)}{2}.$
\end{lemma}

\begin{proof}
Let $v\in V(G)$ be a vertex with $d_G(v)=\Delta(G)$, and let $U\subseteq N_G(v)$ be an odd subset of size $|U|\geq \Delta(G)-1$. Apply Gallai's Theorem to $G[U]$ to obtain a partition $U=V_e\cup V_o$, and observe that $V_o$ must be of an even size (so in particular, $|V_e|$ is odd). If $|V_o|\geq \Delta(G)/2$, then we are done. Otherwise, define $V^*=\{v\}\cup V_e$, and observe that $G[V^*]$ has all its degrees odd and is of size at least $\Delta(G)/2$ as required.
\end{proof}

The next lemma appears as Theorem 1 in \cite{Sco92}, and again, for the sake of completeness, we give its proof here.
\begin{lemma}
\label{lemma2}
 For every graph $G$ with $\delta(G)\ge 1$ we have that
 $f_o(G)\ge \frac{\alpha(G)}{2}.$
\end{lemma}

\begin{proof}
Let $I\subseteq V(G)$ be a largest independent set in $G$.
Since $\delta(G)\geq 1$, every $u\in I$ has at least one neighbor in $V(G)\setminus I$.

Let $D\subseteq V(G)\setminus I$ be a smallest subset dominating all vertices in $I$. Observe that by the minimality of $D$ for every $w\in D$ there exists some $u_w\in I$ such that $N_G(u_w)\cap D=\{w\}$; let $I_D:=\{u_w \mid w\in D\}$.

Let $D'\subseteq D$ be a subset of $D$ chosen uniformly at random, and let $I_0\subseteq I\setminus I_D$ be a subset consisting of all elements $u\in I\setminus I_D$ that have an odd degree into $D'$.

Let
$$I_1=\{u_w\in I_D \mid w\in D' \textrm{ and } w' \textrm{ has even degree in }D'\cup I_0\},$$
and observe that  $G[I_0\cup I_1 \cup D']$ is an induced subgraph of $G$ with all its degrees odd.

Finally, since $\Pr[u\in I_0]=\frac{1}{2}$, by linearity of expectation we have that
$$\mathbb{E}[|I_0\cup I_1 \cup D'|]=\mathbb{E}[|I_0|]+\mathbb{E}[|I_1|]+\mathbb{E}[|D'|]\geq  \frac{|I|-|D|}{2}+\frac{|D|}{2}=\frac{\alpha(G)}{2}.$$
Hence there exists a set $D'$ for which
$$|I_0|+|I_1|+|D'|\geq \frac{\alpha(G)}{2},$$
as desired.
\end{proof}

Next we argue that if $G$ contains a semi-induced matching with ``nice'' expansion properties, then it also has a large induced subgraph with all degrees odd.
\begin{lemma}
\label{lemma3}
Let $G$ be a graph and let $M$ be a matching in $G$ with parts $U$ and $W$, where every vertex $w\in W$ has only one neighbor in $G$ between the vertices covered by $M$. Assume that $|N_G(U)\setminus (W\cup N_G(W))|\ge k$. Then
$f_o(G)\ge \frac{k}{4}.$
\end{lemma}

\begin{proof}
Let $X=N_G(U)\setminus (W\cup N_G(W))$ and recall that $|X|\geq k$. Let $U_0$ be a random subset of $U$ chosen according to the uniform distribution, and let
$$X_0=\{x\in X: d_G(x,U_0)\mbox{ is odd}\}.$$
Since $\mathbb{E}[|X_0|]= |X|/2$, it follows that there exists an outcome $U_0\subseteq U$ for which $|X_0|\ge |X|/2\geq k/2$. Fix such $U_0$.

Next, apply Gallai's theorem to $G[X_0]$ to find a subset $X_1\subseteq X_0$ with $|X_1|\ge |X_0|/2\ge k/4$ and all degrees in $G[X_1]$ even. Finally, for every $u\in U_0$ with $d_G(u,X_1)$ even, add an edge of $M$ containing $u$. Clearly, the obtained graph $G_1$ has size at least $|X_1|\ge |X|/4\ge k/4$, and all its degrees are odd. This completes the proof.
\end{proof}

The following simple lemma will be used several times below.
\begin{lemma}
\label{lemma1n}
Let $G$ be a bipartite graph with parts $A,B$ such that $d(b)>0$ for every $b\in B$. Assume that $|A|\leq \alpha |B|$ for some $0<\alpha\le 1$. Then there is an edge $ab\in E(G)$ with $d(a)\geq \frac{d(b)}{\alpha}$.
\end{lemma}

\begin{proof}
We have:
$$
\sum_{ab\in E(G)}\left(\frac{1}{d(b)}-\frac{1}{d(a)}\right)=\sum_{b\in B}d(b)\cdot\frac{1}{d(b)}-\sum_{a\in A, d(a)>0}d(a)\cdot\frac{1}{d(a)}\ge |B|-|A|\geq (1-\alpha)|B|\,.
$$
Hence there is $b\in B$ with
$$
\sum_{a\in N_G(b)}\left(\frac{1}{d(b)}-\frac{1}{d(a)}\right)\geq 1-\alpha\,.
$$
It follows that there is a neighbor $a$ of $b$ for which $\frac{1}{d(b)}-\frac{1}{d(a)}\geq (1-\alpha)\frac{1}{d(b)}$, implying $d(a)\geq \frac{d(b)}{\alpha}$ as desired.
\end{proof}

For a graph $G=(V,E)$ and $\beta>0$, define
$$
L=L(G;\beta)=\left\{v\in V: \exists u\in V, uv\in E(G), |N(u)\setminus N(v)|\geq \beta|N(u)\cup N(v)|\right\}\,.
$$
We say that for $v\in L$, an edge $uv$ as above {\em witnesses} $v\in L$.

Set
\begin{eqnarray*}
\beta &=& \frac{1}{20}\,,\\
\delta &=& \frac{1}{14}\,,\\
\epsilon &=&\frac{1}{10}\,.
\end{eqnarray*}

The next lemma is a key part in the proof of our main theorem. We did not really pursue the goal of optimizing the constants in its statement.

\begin{lemma}
\label{lemma2n}
Let $G=(V,E)$ be a graph on $|V|=n$ vertices with $\delta(G)>0$ and $|L(G;\beta)|\leq \delta n$. Then $f_o(G)\ge n/61$.
\end{lemma}
\begin{proof}
Define
\begin{eqnarray*}
V_1 &=& \{v\in V\setminus L: d(v,L)\geq \epsilon d(v)\}\,,\\
V_2 &=&V\setminus (V_1\cup L).
\end{eqnarray*}

Suppose first that $ |V_1|\geq 12 |L|$. Observe that $d(v,L)\ge \epsilon d(v)>0$ by the assumption $\delta(G)>0$. By Lemma \ref{lemma1n} there exists $uv\in E(G)$ with $v\in V_1$ and $u\in L$ such that $d(u,V_1)\geq 12 d(v,L)\ge 12\epsilon d(v)$. Therefore we have that
$$|N(u)\setminus N(v)|\geq d(u)-d(v)\ge \frac{12\epsilon-1}{12\epsilon+1}(d(u)+d(v))>\beta |N(u)\cup N(v)|,$$
so in particular $v$ should also be in $L$ with $uv$ witnessing it --- a contradiction. We conclude that
$$
|V_1|< 12 |L|\le 12\delta n\,,
$$
and therefore $|V_2|\geq (1-\delta-12\delta)n=\frac{n}{14}$.

Let $v\not\in L$. Take an edge $uv\in E(G)$. Then
$$
\max\{1,d(u)-d(v)\}\le |N(u)\setminus N(v)|\le \beta |N(u)\cup N(v)|\le \beta (d(u)+d(v))\,,
$$
yielding:
$$
d(u)\le \frac{1+\beta}{1-\beta}d(v)\,,
$$
and
$$
\beta\left(\frac{1+\beta}{1-\beta}+1\right)d(v)\ge 1\,.
$$
This shows that every vertex $v\in V\setminus L$ has degree $d(v)\ge\lceil\frac{1-\beta}{2\beta}\rceil=10$.

Let now $uv\in E(G)$ with $u,v\not\in L$. Then
$$
|N(u)\setminus N(v)|, |N(v)\setminus N(u)|\leq \beta |N(u)\cup N(v)|\,,
$$
and hence
\begin{equation}\label{intersection close to union}
|N(u)\cap N(v)|\geq (1-2\beta)|N(u)\cup N(v)|.
\end{equation}
Since
$$|N(u)\cap N(v)|\leq \min\{d(u),d(v)\} \textrm{ and } |N(u)\cup N(v)|\geq \max\{d(u),d(v)\},$$ it follows that

\begin{equation}\label{degrees are close}
(1-2\beta)d(u)\leq d(v)\leq \frac{d(u)}{1-2\beta}<(1+3\beta)d(u).
\end{equation}

Now, for all $v\notin L$ define $R(v)=\left(\{v\}\cup N(v)\right)\setminus L$. Notice that as $d(v)\ge 10$  we have $|R(v)|\ge (1-\epsilon)d(v)+1\ge 10$ for $v\in V_2$. Suppose that $R(u)\cap R(v)\neq \emptyset$ for some $u\neq v$ where $v\in V_2$ (note that it might be that $u\in V_1$). Then for $w\in R(v)\cap R(u)$, by \eqref{intersection close to union} we have
\begin{equation} \nonumber
|N(u)\cap N(w)|\geq (1-2\beta)|N(u)\cup N(w)| \textrm{ and } |N(v)\cap N(w)|\geq (1-2\beta)|N(v)\cup N(w)|,
\end{equation}

which implies, by the identity $|A\triangle B|=|A\cup B|-|A\cap B|$, that

$$
|N(u)\triangle N(w)|\leq \frac{2\beta}{1-2\beta} |N(u)\cap N(w)|<3\beta |N(u)\cap N(w)|
$$
and
$$
|N(v)\triangle N(w)|\leq\frac{2\beta}{1-2\beta} |N(v)\cap N(w)|<3\beta |N(v)\cap N(w)|\,.
$$

Therefore, we have
\begin{align} \label{chaining}
    |N(u)\cap N(v)|&\geq |N(u)\cap N(v)\cap N(w)|  \\ \nonumber
    &\geq |N(u)\cup N(v)|-|N(u)\triangle N(w)|-|N(v)\triangle N(w)| \\ \nonumber
    &> |N(u)\cup N(v)|-6\beta \max\{d(u),d(v)\}\\\nonumber
    &\geq (1-6\beta)|N(u)\cup N(v)|.
\end{align}

Since  $v\in V_2$ we conclude that
\begin{align} \label{intersection}
     |R(u)\cap R(v)|&\geq |N(u)\cap N(v)|-\epsilon d(v)\\ \nonumber
    &> \left(1-6\beta-\epsilon\right)|N(u)\cup N(v)|\\ \nonumber
    &\geq  \left(1-6\beta-\epsilon\right)|N(u)\cup R(v)|\\ \nonumber
    &=  \left(1-8\beta\right)|N(u)\cup R(v)|\,.
\end{align}

Next, let $R_1,\ldots, R_k$ be a maximal by inclusion collection of non-intersecting sets $R(v_i), v_i\in V_2$. Due to maximality, every $v\in V_2$ has its set $R(v)$ intersecting with at least one of the $R_i$'s; moreover, the above argument shows that it can intersect only one such set. Define now
$$
U_i=\{v\notin L: R(v)\cap R_i\ne\emptyset\}\,.
$$
Trivially we have $R_i\subseteq U_i$. Also, $V_2 \subseteq \bigcup_{i=1}^k U_i $ due to the maximality of the family $R_1,\ldots,R_k$.

We wish to show that all $U_i$ are disjoint and that there are no edges in between different $U_i$'s. (This will add to the above stated fact that the family of $U_i$'s forms a cover of $V_2$.)

To prove the latter claim, suppose that there exists an edge $w_1w_2\in E(G)$ for some $w_1\in U_i,w_2\in U_j$, $1\le i\ne j\le k$. We will obtain a contradiction by showing that $R_i\cap R_j\neq \emptyset$. Since both $w_1,w_2\notin L$, by \eqref{intersection close to union} and \eqref{degrees are close} we conclude that
    $$|N(w_1)\cap N(w_2)|\geq (1-2\beta)|N(w_1)\cup N(w_2)| \textrm{ and } |N(w_1)|\in (1\pm 3\beta)|N(w_2)|\,.$$

Moreover, by \eqref{chaining} we have
$$|N(w_1)\cap N(v_i)|>(1-6\beta)|N(w_1)\cup N(v_i)|,\textrm{ and }|N(w_2)\cap N(v_j)| \geq (1-6\beta)|N(w_2)\cup N(v_j)|.$$

Since $v_i, v_j\in V_2$, the above inequalities imply that
$$|N(w_1)\cap R_i|> (1-6\beta-\epsilon)|N(w_1)\cup R_i|,\textrm{ and }|N(w_2)\cap R_j| > (1-6\beta-\epsilon)|N(w_2)\cup R_j|.$$

It follows that
$$
|N(w_1)\cap R_i|>(1-6\beta-\epsilon)|N(w_1)|
$$
and
$$
|N(w_2)\cap R_j|>(1-6\beta-\epsilon)|N(w_2)|\,,
$$
and recalling that
$$
|N(w_1)\cap N(w_2)|\ge (1-2\beta)|N(w_1)\cup N(w_2)|\,,
$$
we conclude that $R_i\cap R_j\ne\emptyset$ --- a contradiction. In a similar way we can show that $U_i\cap U_j=\emptyset$.

Next, suppose that $|U_i|\ge (1+19\beta)|R_i|$. Then by looking at the auxiliary bipartite graph between $R_i$ and $U_i$ ($v\in R_i$, $u\in U_i$ are connected by an edge if $uv\in E(G)$) and by applying Lemma \ref{lemma1n} to this graph we derive that there are $v\in R_i$, $u\in U_i$ with $d(v)\ge (1+19\beta) d(u,R_i)$. Since $uv\in E(G)$ and both $u,v\not\in L$, it follows that
$$
d(v)<(1+3\beta) d(u)\,.
$$
Moreover, since $u\in U_i$ we have:
$$
d(u,R_i)\ge (1-8\beta)d(u)\,.
$$
All in all, since $d(v)\ge (1+19\beta) d(u,R_i)$ we conclude that
$$
(1+3\beta)d(u)> d(v) \ge (1+19\beta)d(v,R_i)>(1+19\beta)(1-8\beta)d(u)>(1+3\beta)d(u)\,,
$$
a contradiction.

Therefore, we can assume that $|U_i|\leq (1+19\beta)|R_i|$ for all $1\leq i\leq k$. Looking at the induced subgraph $G[U_i]$, we note that it has vertex $v_i$ of degree $|R_i|-1\ge \frac{9|R_i|}{10}\ge \frac{9}{10(1+19\beta)}|U_i|$.
By applying Lemma \ref{lemma1} to $G[U_i]$ we find an induced odd subgraph $O_i$ of $G[U_i]$ of size at least $\frac{9}{20(1+19\beta)}|U_i|=\frac{9|U_i|}{39}$.

Finally, since all $U_i$'s are disjoint, there are no edges between any two such $U_i$'s and since $V_2\subseteq \bigcup U_i$, we conclude that $O=\bigcup_{i=1}^k O_i$ is an induced odd subgraph of size at least $\frac{9|V_2|}{39}> n/61$. This completes the proof. \end{proof}

\section{Proof of Theorem \ref{main}}
The main plan is as follows. We will grow edge by edge a matching $M$ with sides $U,W$ so that every $w\in W$ has exactly one neighbor between the vertices covered by $M$ (which is of course its mate $u$ in the matching). Moreover, the set $U$ has ``many'' neighbors outside of $M$ not connected to $W$. If the set of such neighbors is substantially large, then we will be able to apply Lemma \ref{lemma3} to get a large induced subgraph with all degrees odd. Otherwise we will show that either there exists a large subset of vertices $V'$ such that $\delta(G[V'])\geq 1$ with small $L(G[V'];1/20)$ (and then we are done by Lemma \ref{lemma2n}), or that we can extend the matching while enlarging substantially the set of neighbors outside $M$ not connected to $W$. The details are given below.

We start with $M_0=\emptyset$, and given $M_i$, $i\geq 0$, we define
\begin{eqnarray*}
X_i &=& N(U_i)\setminus (W_i\cup N(W_i))\,,\\
V_i &=& V\setminus N(U_i\cup W_i)\,.
\end{eqnarray*}

In particular, we initially have $X_0=\emptyset$ and $V_0=V$. We will run our process until the first time we have $|V_i|<n/2$ (in particular, we may assume throughout the process that $|V_i|\geq n/2$). Now, fix $\beta=1/20$ and $\delta=1/14$ (same parameters as set before Lemma \ref{lemma2n}). Our goal is to show that $f_o(G)\geq \frac{n}{T}$, where $T=10000$. We will maintain $|X_i|\geq \frac{|V\setminus V_i|}{40}$. If at some point we reach $|X_i|\geq \frac{4 n}{T}$ then we are done by Lemma \ref{lemma3}. Hence we assume $|X_i|\leq \frac{4n}{T}=\frac{n}{2500}$. Moreover, if $G[V_i]$ has at least $2n/T$ isolated vertices, then since this set induces an independent set in $G$, by Lemma \ref{lemma2} we are done as well. Therefore, letting $V'_i\subseteq V_i$ be the set of all non-isolated vertices in $G[V_i]$, since $|V_i|\geq n/2$ we obtain that $|V'_i|\geq (1-4/T)|V_i|\geq |V_i|/2$. We can further assume $|L(G[V'_i];\beta)|\geq \delta |V'_i|\geq \delta n/4$, as otherwise by Lemma \ref{lemma2n} we obtain an odd subgraph of size at least $|V'_i|/61\geq n/244$. Our goal now is to show that under these assumptions we can add an edge to $M_i$ while maintaining $|X_{i+1}|\geq \frac{|V\setminus V_{i+1}|}{40}$.

Consider first the case where every $v\in L:=L(G[V_i'];\beta)$ satisfies $d(v,X_i)\geq d(v,V_i)/40$. By Lemma \ref{lemma1n} applied to the bipartite graph between $X_i$ and $L$, using the fact that
$$|X_i|\leq \frac{4n}{T}= \frac{n}{2500}\leq \frac{|L|}{44}\,,$$
we derive that there is an edge $xv$ with $x\in X_i$ and $v\in L$ and $d(x,L)\geq 44d(v,X_i)\geq 1.1d(v,V_i)>0$. Then we can define $M_{i+1}$ by adding $xv$ to $M_{i}$ and setting $U_{i+1}:=U_i\cup \{x\}$ and $W_{i+1}:=W_i\cup \{v\}$. By doing so we obtain that
\begin{align*}
    |X_{i+1}|&=|N(U_{i+1})\setminus (W_{i+1}\cup N(W_{i+1})|\\
    &\geq |N(U_{i})\setminus (W_i\cup N(W_{i}))|+|N(x,V_i)|-|N(v,X_i)|-|N(v,V_i)|\\
    &= |X_i|+d(x,V_i)-d(v,X_i)-d(v,V_i)\\
    &\geq |X_i|+d(x,V_i)\left(1-\frac{1}{44}-\frac{10}{11}\right)\\
    &> |X_i|+\frac{3d(x,V_i)}{44}.
\end{align*}
Moreover, since we clearly have that
$$|V_{i+1}|\geq |V_i|-d(x,V_i)-d(v,V_i)\geq |V_i|-\frac{21d(x,V_i)}{11},$$
it follows that at least$\frac{3/44}{21/11}>\frac{1}{40}$ proportion of the vertices deleted from $V_i$ go to $X_{i+1}$.

In the complementary case there exists a vertex $v\in L$ with $d(v,X_i)\leq d(v,V_i)/40$. Let $uv$ be an edge in $G[V_i']$ witnessing $v\in L$ (that is, $|N(u,V_i)\setminus N(v,V_i)|\geq \beta |N(u,V_i)\cup N(v,V_i)|$). Then we can define $M_{i+1}$ by adding $uv$ to $M_i$, and set $U_{i+1}:=U_i\cup \{u\}$ and $W_{i+1}:=W_i\cup \{v\}$. In this case we have:
\begin{align*}
    |X_{i+1}|&=|N(U_{i+1})\setminus (W_{i+1}\cup N(W_{i+1}))|\\
    &\geq |N(U_{i})\setminus (W_i\cup N(W_{i}))|+|N(u,V_i)\setminus N(v,V_i)|-|N(v,X_i)|\\
    &= |X_i|+|N(u,V_i)\setminus N(v,V_i)|-|N(v,X_i)|\\
    &\geq |X_i|+\beta |N(u,V_i)\cup N(v,V_i)|-|N(v,X_i)|\\
    &\geq |X_i|+(\beta-\frac{1}{40}) |N(u,V_i)\cup N(v,V_i)|.
\end{align*}

Moreover, since we have $|V_{i+1}|\geq |V_i|-|N(u,V_i)\cup N(v,V_i)|$, at least $\beta - \frac{1}{40}=\frac{1}{40}$ proportion of the vertices deleted from $V_i$ go to $X_{i+1}$.

All in all, in each step, either we find an odd subgraph of size at least $\frac{n}{T}$ (in case that we have ``many'' isolated vertices, or that $|X_i|\geq \frac{4n}{T}$, or that $L(G[V'_i];\beta)$ is ``large''), or we can keep $X_i$ of size at least $\frac{|V\setminus V_i|}{40}$. In particular, if the latter case holds until $|V_i|<n/2$, we obtain that $|X_i|\geq \frac{n}{80}$ and we are done by Lemma \ref{lemma3}. This completes the proof.\hfill$\Box$

\bigskip

{\bf Acknowledgement.}
We would like to thank Alex Scott for his remarks,  and for pointing out a serious flaw in the previous version.

\end{document}